\documentclass[11pt]{amsart}

\usepackage{amsmath}
\usepackage{amssymb}
\usepackage{amsthm}
\usepackage{graphicx}
\usepackage{verbatim}
\usepackage{tikz}
\usetikzlibrary{calc}
\usetikzlibrary{decorations.pathreplacing}
\usepackage{enumerate}

\usepackage[margin=1.2in]{geometry}
\def\E{\mathbb{E}}

\def\E{\mathbb{E}}

\def\eps{\epsilon}

\def\gam{\gamma}

\def\1{\mathbf{1}}

\def\lam {\lambda}
\def\Lam{\Lambda}
\def\tce{t_c + \eps}
\def\tce2{t_c + \frac{\eps}{2}}

\newcommand\numberthis{\addtocounter{equation}{1}\tag{\theequation}}
\newcommand*{\MK}[1]{M_{K_{#1,#1}}}

\newtheorem*{theorem*}{Theorem}
\newtheorem{theorem}{Theorem}
\newtheorem{lemma}{Lemma}

\newtheorem*{prop*}{Proposition}
\newtheorem{conj}{Conjecture}
\newtheorem{claim}{Claim}

\begin{document}
\title{Independent Sets, Matchings, and Occupancy Fractions}
\author{Ewan Davies}
\author{Matthew Jenssen}
\author{Will Perkins}
\thanks{WP supported in part by EPSRC grant EP/P009913/1.}
\author{Barnaby Roberts}
\address{Department of Mathematics, London School of Economics, London, UK.}
\email{\{e.s.davies,m.o.jenssen,b.j.roberts\}@lse.ac.uk }
\address{School of Mathematics, University of Birmingham, Birmingham, UK.}
\email{w.f.perkins@bham.ac.uk}
\date{\today}

\begin{abstract}
We prove tight upper bounds on the logarithmic derivative of the independence and matching polynomials of $d$-regular graphs. For independent sets, this theorem is a strengthening of the results of Kahn, Galvin and Tetali, and Zhao showing that a union of copies of $K_{d,d}$ maximizes the number of independent sets and the independence polynomial of a $d$-regular graph.  

For matchings, this shows that the matching polynomial and the total number of matchings of a $d$-regular graph are maximized by a union of copies of $K_{d,d}$. Using this we prove the asymptotic upper matching conjecture of Friedland, Krop,  Lundow, and Markstr{\"o}m.

In probabilistic language, our main theorems state that for all $d$-regular graphs and all $\lambda$, the occupancy fraction of the hard-core model and the edge occupancy fraction of the monomer-dimer model with fugacity $\lambda$ are maximized by $K_{d,d}$.  Our method involves constrained optimization problems over distributions of random variables and applies to all $d$-regular graphs directly, without a reduction to the bipartite case. 
\end{abstract}

\maketitle
\thispagestyle{empty}

\section{Independent Sets}

Let $G$ be a graph.  The independence polynomial of $G$ is
\[P_G(\lam) = \sum_{I \in \mathcal I} \lam^{|I|} \]
where $\mathcal I $ is the set of all independent sets of $G$.  By convention we consider the empty independent set to be a member of $\mathcal I$.  The \textit{hard-core model} with fugacity $\lam$ on $G$ is a random independent set $I$ drawn according to the distribution
\[ 
\Pr_\lam[ I] = \frac{ \lam^{|I|}} { P_G(\lam)}\,.
\]
 $P_G(\lam)$ is also  called the partition function of the hard-core model on $G$.
 
  In the hard-core model, the quantity $\alpha_G(\lam) = \frac{1}{|V(G)|} \frac{\lam P^\prime_G(\lam) }{P_G(\lam)}$ is the \textit{occupancy fraction}: the expected fraction of vertices of $G$ belonging to the random independent set $I$.  In particular,
\begin{align*}
\label{eq:occdef}
\alpha_G(\lam) &= \frac{1}{|V(G)|} \sum_{v \in G} \Pr[v \in I]  = \frac{1}{|V(G)|} \frac{\sum _{I \in \mathcal I} |I| \lam ^{|I|}   }{P_G(\lam)  } \\
&= \frac{1}{|V(G)|} \frac{\lam P_G^\prime(\lam)}{P_G(\lam) }  = \frac{\lam}{|V(G)|} \left( \log P_G(\lam)  \right )^\prime  \,.
\end{align*}

We write $K_{d,d}$ for the complete bipartite graph with $d$ vertices in each part.  If $2d$ divides $n$, let $H_{d,n}$ denote the $d$-regular, $n$-vertex graph that is the disjoint union of $n/(2d)$ copies of $K_{d,d}$. Kahn \cite{kahn2001entropy} showed that $H_{d,n}$ maximizes the total number of independent sets over all $d$-regular, $n$-vertex bipartite graphs, and then showed~\cite{kahn2002entropy} that in fact $K_{d,d}$ (or $H_{d,n}$) maximizes $\frac{1}{|V(G)|} \log P_G(\lam)$ for $\lam \ge 1$ over all $d$-regular bipartite graphs. The log partition function result generalizes the counting result as the latter can be recovered by setting $\lam=1$. Galvin and Tetali~\cite{galvin2004weighted} then gave a broad generalization of Kahn's result to counting homomorphisms from a $d$-regular, bipartite $G$ to any graph $H$.  The case of $H$ formed of two connected vertices, one with a self-loop, is that of counting independent sets. Via a modification of $H$ and a limiting argument, they proved that in fact $\frac{1}{|V(G)|} \log P_G(\lam)$ is maximized for any $\lam>0$ over $d$-regular bipartite graphs by $K_{d,d}$.  Zhao \cite{zhao2010number} then removed the bipartite restriction in these results for  independent sets by reducing the general case to the bipartite case, in particular proving that $H_{d,n}$ has the greatest number of independent sets of any $d$-regular graph on $n$ vertices.

Here we prove a strengthening of the above results for independent sets. 
\begin{theorem}
\label{thm:occupy}
For all $d$-regular graphs $G$ and all $\lam>0$, we have
\[ \alpha_G(\lam) \le \alpha_{K_{d,d}}(\lam) = \frac{\lam(1+\lam)^{d-1}}{2(1+\lam)^d -1} \,. \]
  The maximum is achieved only by unions of copies of $K_{d,d}$. 
  \end{theorem}

In particular Theorem~\ref{thm:occupy} states that the derivative of $\frac{1}{|V(G)|}\log P_G(\lam)$ is maximized over $d$-regular graphs for all $\lam$ by $K_{d,d}$, which when integrated, immediately implies that the normalized log partition function, $\frac{1}{|V(G)|} \log P_G(\lam)$, is maximized. Even more, it says that the difference $\frac{1}{2d} \log P_{K_{d,d}}(\lam)- \frac{1}{|V(G)|} \log P_G(\lam)$ is strictly increasing in $\lam$ for any $d$-regular graph $G$ that is not $H_{d,n}$.  Note that $\frac{1}{n} \log P_{H_{d,n}}(\lam) = \frac{1}{2d} \log P_{K_{d,d}}(\lam)$ for any $n$ divisible by $2d$.

In Section \ref{sec:givensize} we observe that the above bound on the partition function gives new upper bounds on a related problem: maximizing the number of independent sets of a given size in $d$-regular graphs. 

Next, let $\alpha_{T_d}(\lam)$ be occupancy fraction of the unique translation invariant hard-core measure on the infinite $d$-regular tree $T_d$ at fugacity $\lam$; that is, $\alpha_{T_d}(\lam)$ is the solution of the equation
\[ \frac{\alpha}{\lam (1-\alpha)} = \left(\frac{1-2\alpha}{1-\alpha} \right)^d  \]
(see e.g. \cite{bhatnagar2014decay}).

Using a variant of the method used to establish Theorem \ref{thm:occupy}, we  prove a lower bound on the occupancy fraction in any $d$-regular, vertex-transitive, bipartite graph $G$.

\begin{theorem}
\label{thm:lowerbound}
For any $d$-regular, vertex-transitive, bipartite graph $G$,
\[ \alpha_G(\lam) > \alpha_{T_d}(\lam)\,.  \]
\end{theorem}

The corresponding statement for the normalized log partition function (the integrated version of Theorem~\ref{thm:lowerbound}) holds without the condition of vertex transitivity~\cite{ruozzi2012bethe}.  Theorem~\ref{thm:lowerbound} itself may not hold without vertex transitivity (see Section 5 of~\cite{csikvari2014matchings} for a related discussion about matchings). For $\lam \le \lam_c(T_d) = \frac{(d-1)^{d-1}}{(d-2)^d}$ (the uniqueness threshold of the hard-core model on $T_d$), the  bound in Theorem~\ref{thm:lowerbound} is asymptotically tight for this class of graphs. From the results of Weitz \cite{weitz2006counting}, any sequence of graphs $G_n$ that converges locally (in the sense of Benjamini-Schramm \cite{benjamini2011recurrence}) to $T_d$ has occupancy fraction $\alpha_{T_d}(\lam) +o(1)$ as $n \to \infty$; for example, we can take a sequence of bipartite Cayley graphs of large girth. 

\section{Matchings}

The matching polynomial of a graph $G$ is
\[ M_G(\lam) = \sum_{H \in \mathcal M} \lam^{|H|}  \]
where $\mathcal M$ is the set of all matchings of $G$ (including the empty matching) and $|H|$ is the number of edges in the matching $H$. Just as in the hard-core model above we can define a probability distribution over matchings:
\[ \Pr_\lam [H] = \frac{ \lam^{|H|} }{M_G(\lam) } \,. \]

This defines the monomer-dimer model from statistical physics \cite{heilmann1972theory}: dimers are  edges of the random matching $H$ and monomers the unmatched vertices.  

The \textit{edge occupancy fraction}, or the dimer density, is the expected fraction of the edges of $G$ in such a random matching:
\begin{equation*}
\alpha^M_G(\lam) = \frac{1}{|E(G)|} \sum_{e \in G} \Pr[e \in H] = \frac{1}{|E(G)|} \frac{\lam M_G^\prime(\lam)}{M_G(\lam)} \,.
\end{equation*}

Our next result is  an upper bound on the edge occupancy fraction of any $d$-regular graph:

\begin{theorem}
\label{thm:matchingGeneral}
For all $d$-regular graphs $G$ and all $\lam>0$, we have
\[ \alpha^M_G(\lam) \le \alpha^M_{K_{d,d}}(\lam) \,. \]
The maximum is achieved only by unions of copies of $K_{d,d}$.  
\end{theorem}

This states that the normalized logarithmic derivative of $M_G(\lam)$ is maximized by $K_{d,d}$, and hence via integration that $K_{d,d}$ (and thus also $H_{d,n}$) maximizes $\frac{1}{|E(G)|} \log M_G(\lam)$ for any $\lam >0$. This resolves Conjecture 7.1 in \cite{galvin2014three}. Br\'egman's theorem \cite{bregman1973some} says that the number of perfect matchings of a $d$-regular, $n$-vertex bipartite graph is maximized by $H_{d,n}$, and this was extended by Kahn and Lov\'asz to all $d$-regular graphs (see \cite{galvin2014three} for a full discussion). Our result on $M_G(\lam)$ extends this: letting $\lam \to \infty$ recovers the result for perfect matchings, while setting $\lam =1$ shows that $H_{d,n}$ maximizes the total number of matchings  of any $d$-regular graph on $n$ vertices.

In Section \ref{sec:givensize} we use Theorem~\ref{thm:matchingGeneral} to give new upper bounds on the number of matchings of a given size in $d$-regular graphs. We then use these bounds to prove the `asymptotic upper matching conjecture' of Friedland, Krop, Lundow, and Markstr{\"o}m~\cite{friedland2008validations}.

\section{Related work}

The results of Kahn \cite{kahn2001entropy,kahn2002entropy}, Galvin and Tetali \cite{galvin2004weighted}, and  Zhao \cite{zhao2010number} culminating in the fact that $\frac{1}{|V(G)|} \log P_G(\lam)$ is maximized over $d$-regular graphs by $K_{d,d}$ are based on the entropy method, a powerful tool for the type of problems we address here. Apart from the results mentioned above, see \cite{radhakrishnan20036} and \cite{galvin2014three} for surveys of the method. A direct application of the method requires the graph $G$ to be bipartite. Zhao \cite{zhao2011bipartite} showed that in some, but not all cases, this restriction can be removed by using a `bipartite swapping trick'.  An entropy-free proof of Galvin and Tetali's general theorem on counting homomorphisms was recently given by Lubetzky and Zhao~\cite{lubetzky2014replica}.  Our method also does not use entropy, but in contrast to the other proofs it works directly for all $d$-regular graphs, without a reduction to the bipartite case. The method deals directly with the hard-core model instead of counting homomorphisms and seems to require more problem-specific information than the entropy method; a question for future work is to extend the method to a more general class of homomorphisms.  

The technique of writing the expected size of an independent set in two ways (as we do here) was used by Shearer \cite{shearer1995independence} in proving lower bounds on the average size of an independent set in $K_r$-free graphs and then by Alon \cite{alon1996independence} for graphs in which all vertex neighborhoods are $r$-colorable. The idea of bounding the occupancy fraction instead of the partition function comes in part from work of the third author \cite{perkins2015birthday}  in improving, at low densities, the bounds on matchings of a given size in Ilinca and Kahn \cite{ilinca2013asymptotics}  and independent sets of a given size in Carroll, Galvin, and Tetali \cite{carroll2009matchings}. The use of linear programming for counting graph homomorphisms appears in Kopparty and Rossman \cite{kopparty2011homomorphism}, where they use a combination of entropy and linear programming  to compute a related quantity, the homomorphism domination exponent, in chordal and series-parallel graphs.

For matchings, Carroll, Galvin, and Tetali~\cite{carroll2009matchings} used the entropy method to give an upper bound of $\frac{1}{2} \log (1+d \lam)$ on $\frac{1}{|V(G)|} \log M_G(\lam)$ over $d$-regular graphs. It was previously conjectured (eg.\ \cite{friedland2008number,galvin2014three}) that  $K_{d,d}$ maximizes $\frac{1}{|V(G)|} \log M_G(\lam)$ over all $d$-regular graphs. This is an implication of our Theorem~\ref{thm:matchingGeneral}.

In \cite{csikvari2014lower}, Csikv{\'a}ri proved  the `lower matching conjecture' of \cite{friedland2008number} and in \cite{csikvari2014matchings} gave a new lower bound on the number of perfect matchings of $d$-regular, vertex-transitive, bipartite graphs, in both comparing an arbitrary graph with the infinite $d$-regular tree (see also the recent extension by Lelarge \cite{lelarge2015counting} to  irregular graphs). Proposition 2.10 in \cite{csikvari2014matchings} states that the edge occupancy fraction of any $d$-regular, vertex-transitive, bipartite graph is at least that of the infinite $d$-regular tree; in Theorem~\ref{thm:lowerbound} we prove an analogous result for independent sets.  Csikv{\'a}ri's techniques in the two papers are different than the methods of this paper, but similar in that he bounds the occupancy fraction instead of directly working with the partition function.  His results rely on an elegant interplay between the Heilman-Lieb theorem \cite{heilmann1972theory} and Benjamini-Schramm convergence of bounded-degree graphs.   

In statistical physics, the analogue of the occupancy fraction in a general spin system is called the \textit{mean magnetization}; on general graphs it is $\#P$-hard to compute the magnetization in the ferromagnetic Ising model, the monomer-dimer model, and the hard-core model~\cite{sinclair2014lee,schulman2015symbolic}. 

\section{The Method}
\label{sec:themethod}
To introduce our method, we start by proving Theorem~\ref{thm:occupy} under the assumption that $G$ is triangle-free.  In what follows $I$  will denote the random independent set drawn according to the hard-core model with fugacity $\lam$ on a $d$-regular, $n$-vertex graph $G$.

We say a vertex $v$ is \textit{occupied} if $v \in I$ and \textit{uncovered} if none of its neighbors are in $I$: $N(v) \cap I = \emptyset$. Let $p_v$ be the probability $v$ is occupied and $q_v$ be the probability $v$ is uncovered.  The idea of considering $q_v$ appears in Kahn's paper \cite{kahn2001entropy}. 

 We will show that for every $\lam>0$ and any triangle-free $G$, $\alpha_G(\lam)$ is maximized by $K_{d,d}$.  (It is easy to see by linearity of expectation or by manipulating the partition function that the occupancy fraction is the same for any number of copies of $K_{d,d}$).  

Letting $\alpha=\alpha_G(\lam)$, we write
\begin{align}
\nonumber
\alpha & =  \frac{1}{n} \sum_{v \in G} p_v \\
\label{eq:pvqv}
&= \frac{1}{n} \sum_{v \in G} \frac{\lam}{1+ \lam} q_v \\
\label{eq:uncov}
&= \frac{\lam}{1+ \lam} \cdot \frac{1}{n} \sum_{v \in G} \sum_{j=0}^d \Pr[j \text{ neighbors of } v \text{ are uncovered}] \cdot (1+\lam)^{-j}  \\
\nonumber
&= \frac{\lam}{1+ \lam} \cdot \E [ (1+\lam)^{-Y}]
\end{align}
where $Y$ is the random variable that counts the number of uncovered neighbors of a uniformly chosen vertex from $G$, with  respect to the random independent set $I$.  $Y$ is an integer valued random variable bounded between $0$ and $d$.  
Equation (\ref{eq:pvqv}) follows since $v$ must be uncovered if it is to be occupied, and conditioning on being uncovered $v$ is occupied with probability $\frac{\lam}{1+ \lam}$.  Equation (\ref{eq:uncov}) is similar: conditioned on the event that $u_1, \dotsc, u_j$, neighbors of $v$, are all uncovered, the probability that none are occupied is $(1+\lam)^{-j}$. 
This is where we use the fact that $G$ is triangle-free: there are no edges between neighbors of $v$.

We also have
\begin{align*}
\E Y &= \frac{1}{n} \sum_{v \in G} \sum_{u \sim v} q_u = d \cdot \frac{1+ \lam}{\lam} \alpha
\end{align*}
since each $u$ appears in the double sum exactly $d$ times as $G$ is $d$-regular.  This gives the identity
\begin{equation*}
\E Y = d \cdot \E [ (1+\lam)^{-Y}]\,.
\end{equation*}
Now let 
\[ \alpha^* = \frac{\lam}{d (1+ \lam)} \cdot  \sup_{0 \le Y \le d} \{   \E Y: \E Y = d \cdot \E [ (1+\lam)^{-Y}] \}  \]
where the sup is over all distributions of random variables $Y$ bounded between $0$ and $d$.  

For any $\lam$ and $d$ there is a unique distribution $Y$ supported only on $0$ and $d$ that satisfies the constraint $\E Y = d \cdot \E [ (1+\lam)^{-Y}]$.  We claim that the sup is uniquely achieved by this distribution. The claim follows from convexity, but we defer details to the proof of a more general statement in Section \ref{sec:triangles}. Since the distribution $Y$ associated to $H_{d,n}$ satisfies the constraint and is supported on $0$ and $d$, it must maximize $\alpha$.  Since unions of copies of $K_{d,d}$ are the only graphs whose associated distribution is supported on $0$ and $d$, they uniquely achieve the maximum. 

To recap, the method is the following:
\begin{enumerate}[(i)]
\item Define a random variable $Y$ using randomness in the hard-core model on $G$ and in choosing a random vertex of $G$. In the proof above, $Y$ was the number of uncovered neighbors of a random vertex.
\item Write $\alpha$ as the expectation of a function of $Y$.
\item Add constraints that the random variable $Y$ must satisfy for any graph $G$ in our class.  In the case above, the constraints were that the two ways of writing $\alpha$ are equal and that $0 \le Y \le d$.
\item Relax the optimization problem from random variables $Y$ induced by graphs to all random variables $Y$ that satisfy the constraints. Show that the unique maximizer of $\alpha$ is the distribution associated to the extremal graph, and therefore $\alpha$ is maximized by the extremal graph.  
\end{enumerate}

In Section~\ref{sec:triangles} we give the full proof of Theorem~\ref{thm:occupy}.  We prove the lower bound, Theorem~\ref{thm:lowerbound}, in Section~\ref{sec:LB}.  We turn to matchings and Theorem~\ref{thm:matchingGeneral} in Section~\ref{sec:matchingproofs} before giving new bounds on the number of independent sets and matchings of a given size in Section~\ref{sec:givensize}.

\section{Proof of Theorem~\ref{thm:occupy}}
\label{sec:triangles}

For a vertex $v \in G$ and an independent set $I$, we define the \textit{free neighborhood} of $v$ to be the subgraph of $G$ induced by the neighbors of $v$ which are not adjacent to any vertex in $I \setminus N(v)$. We use the convention $v \notin N(v)$.  The vertices in the free neighborhood may be uncovered or covered, but if they are covered it must be from another vertex in the free neighborhood.  In a triangle-free graph the free neighborhood is always a set (possibly empty) of isolated vertices.  Note that if $v \in I$, then the free neighborhood of $v$ is necessarily empty.  

Let $ C$ be the random free neighborhood of $v$ when we draw $I$ according to the hard-core model and choose vertex $v$ uniformly at random from $G$.  For any graph $F$, let $p_F$ be the probability that $ C$ is isomorphic to $F$.  Also let $P_C = P_C(\lam)$ be the independence polynomial of $C$ at fugacity $\lam$.  Then we can write $\alpha$ in two ways:
\begin{align}
\label{eq:tri1}
\alpha  &= \frac{\lam}{1+ \lam}  \E \left [ \frac{1}{P_C(\lam)} \right ] 
\intertext{and}
\label{eq:tri2}
 \alpha  &= \frac{\lam}{d} \E \left[  \frac{P_C^\prime(\lam)}{P_C(\lam)} \right] 
 \end{align} 
where in both equations the expectations are over the random free neighborhood $C$. Equation (\ref{eq:tri1}) follows since $v$ itself is uncovered if and only if all vertices in its free neighborhood are unoccupied. Given that the free neighborhood is isomorphic to $C$, all vertices in the free neighborhood are unoccupied with probability $\frac{1}{P_C(\lam)}$.  Equation (\ref{eq:tri2}) follows by counting the expected number of occupied neighbors of $v$ and dividing by $d$: only vertices in the free neighborhood can be occupied, and, given $C$, the expected number of occupied vertices in the free neighborhood is  $\frac{\lam P_C^\prime(\lam)}{P_C(\lam)}$. 

Now let 
\begin{equation}\label{eq:Cconstr}
\alpha^* = \frac{\lam}{1+ \lam} \cdot \sup \left \{  \E \left [ \frac{1}{P_C(\lam)} \right ]  : \frac{d}{1+\lam} \cdot \E \left [ \frac{1}{P_C(\lam)} \right ] = \E \left[  \frac{P_C^\prime(\lam)}{P_C(\lam)} \right]   \right \}
\end{equation}
where the sup is over all distributions of random free neighborhoods $C$ supported on graphs of at most $d$ vertices. From (\ref{eq:tri1}) and (\ref{eq:tri2}), the distribution obtained from $G$ satisfies the constraint above.  

We claim that for any $\lam>0$, $\alpha^*$ is achieved uniquely by a distribution supported only on the empty graph and the graph consisting of $d$ isolated vertices, $\overline {K_d}$.  The theorem follows since  disjoint unions of copies of $K_{d,d}$ are the only graphs for which the free neighborhood can only be the empty set or $\overline {K_d}$, and since there is a unique distribution with this support satisfying the constraint.  To prove this claim we use the language of linear programming, see e.g.~\cite{boyd2004convex}.

\subsection{The linear program}
 Let $p_C$ be the probability of a given free neighborhood $C$, and let $\mathcal C_d$ be the set of all graphs on at most $d$ vertices, including the empty graph.  Equation (\ref{eq:Cconstr}) defines a linear program with the decision variables $\{ p_C \} _{C \in \mathcal C_d}$.  We write the linear program in standard form as
 \begin{align*}
\alpha^* = \max \,   \frac{\lam}{2(1+\lam)} &\sum_{C \in \mathcal C_d} p_C ( a_C +b_C)   \text{ subject to } \\
&\sum_{C \in \mathcal C_d} p_C = 1 \\
&\sum_{C \in \mathcal C_d} p_C( a_C -b_C) =0 \\
&p_C \ge 0 \, \, \,  \forall C \in \mathcal C_d 
\end{align*}
where $a_C =  \frac{1}{ P_C(\lam)}$ and $b_C = \frac{(1+\lam) P^\prime_C(\lam)}{d P_C(\lam)}$. We can calculate $a_{\emptyset} = 1$, $b_{\emptyset} = 0$, $a_{\overline{K}_d} = (1+\lam)^{-d}$, $b_{\overline{K}_d} = 1$. 
The solution $p_{\emptyset} = \frac{1-(1+\lam)^{-d}}{2 -(1+ \lam)^{-d}} $ and $p_{\overline {K_d}} = \frac{1}{2 -(1+ \lam)^{-d}}$ is the unique feasible solution supported only on $\emptyset $ and $\overline{K}_d$, and gives the objective value $\frac{ \lam (1+\lam)^{d-1}}{2 (1+\lam)^d -1}$.  Our claim is that this is the unique maximum.  

The dual linear program is
\begin{align*}
\alpha^* =  &\min \frac{\lam}{2(1+\lam)} \Lambda_1 \text { s.t. } \\
&\Lambda_1 + \Lambda_2 (a_C - b_C)  \ge a_C + b_C \, \, \, \forall C \in \mathcal C_d
\end{align*}
where $\Lambda_1, \Lambda_2$ are the decision variables.  

Guided by the candidate solution above we set $\Lambda_1 = \frac{2}{2- (1+\lam)^{-d}}$ and $\Lambda_2= 1-\Lambda_1$. With these values, the dual constraints corresponding to $C = \emptyset, \overline{K}_d$ hold with equality, and the objective value is $\frac{\lam}{2(1+\lam)} \Lambda_1=\frac{ \lam (1+\lam)^{d-1}}{2 (1+\lam)^d -1}$.   To finish the proof we claim that $\Lambda_1, \Lambda_2$ are feasible for the dual program, which means showing that 
\[ \Lambda_1 + \Lambda_2 (a_C - b_C)  \ge a_C + b_C  \]
for all $C \in \mathcal C_d$. We will show that in fact the inequality holds strictly for all $C \in \mathcal C_d \setminus \{\emptyset, \overline {K}_d \}$.  Substituting our values of $\Lam_1, \Lam_2$, this inequality reduces to 
\begin{equation}
\label{eq:indFact}
 \frac{\lam P_C^\prime(\lam)}{P_C(\lam)-1} < \frac{ \lam d (1+\lam)^{d-1}}{(1+\lam)^d- 1} \, .
\end{equation}
The LHS of (\ref{eq:indFact}) is the expected size of the random independent set from the hard-core model on $C$ conditioned on it being non-empty.  The RHS is the same quantity for $\overline{K}_d$.

Inequality \eqref{eq:indFact} follows directly from the observation that, over all $C\in \mathcal C_d$, the graph $\overline{K}_d$ maximizes the ratio of subsequent terms in the polynomial $P_C$. 
Let $t_i=\binom{d}{i}$, the coefficient of $\lam^i$ in $P_{\overline{K}_d}$, and write $P_C= 1+\sum_{i=1}^d r_i\lam^i$. 
We have $(i+1)t_{i+1}=(d-i)t_i$ and $(i+1)r_{i+1}\leq(d-i)r_i$ by counting independent sets of size $i+1$. 

To verify \eqref{eq:indFact} we show that for each $1\leq k \leq d$ the coefficient $c_k$ of $\lam^k$ in the polynomial  $(\lam P'_{\overline{K}_d})(P_C-1) - (\lam P'_C)(P_{\overline{K}_d}-1)$ is non-negative. 
We have
\begin{align*}
s_k &= \sum_{i=1}^{k-1}it_ir_{k-i} - \sum_{i=1}^{k-1} it_{k-i}r_i\\
&= \sum_{i=1}^{\lfloor k/2\rfloor } (k-2i)(t_{k-i}r_i - t_ir_{k-i})\,.
\end{align*}
Observe that term-by-term the above sum giving $s_k$ is non-negative by comparing the ratio of successive coefficients in $P_{\overline{K}_d}$ and $P_C$. 
Furthermore, if $P_C\neq P_{\overline{K}_d}$ then at least one $s_k$ must be positive, which completes the claim.

To see the optimizer is unique note that strict inequality in the dual constraints corresponding to configurations besides $\emptyset$ and $\overline{K}_d$ implies by complementary slackness that any optimal solution is supported on these two configurations, and there is a unique such distribution.

\section{Proof of Theorem~\ref{thm:lowerbound}}
\label{sec:LB}

To prove Theorem~\ref{thm:lowerbound} we will use the fact that occupancies of vertices on the same side of a bipartite graph are positively correlated:

\begin{lemma}
\label{lem:bipariteCor}
Let $G$ be a bipartite graph with bipartition $\mathcal E \cup \mathcal O$.  For any $r \ge 2$, let $u_1, u_2, \dots , u_r \in \mathcal E$.  Then
\[ \Pr[ \{ u_1 , \dots , u_r \}  \subseteq I] \ge \prod_{i=1}^r p_{u_i} \]
in the hard-core model for any $\lambda$.  Similarly, let $U$ be the random set of uncovered vertices of $G$.  Then 
\[ \Pr[ \{ u_1 , \dots , u_r \}  \subseteq U] \ge \prod_{i=1}^r q_{u_i} \]
Moreover, the inequalities are strict when $\lam >0$ and at least two of the $u_i$'s are in the same connected component of $G$.  
\end{lemma}

The first part of the lemma follows by induction on $r$ from the fact that $\Pr[ u_1, u_2 \in I] > \Pr[u_1 \in I] \cdot \Pr[u_2 \in I]$ when $u_1, u_2$ are in the same connected component and in the same part of the bipartition of $G$. In~\cite{van1994percolation} this is shown to be a consequence of the FKG inequality; see also~\cite{fill2001stochastic} and Corollary 1.5 of~\cite{bencs2014christoffel}.  An intuitive reason for this fact (which can be turned into a rigorous argument using Weitz's tree~\cite{weitz2006counting}), is that conditioning on the event that a vertex $v$ is occupied forbids its neighbors from being in the independent set; conditioning on the event that $v$ is not occupied increases the probability each of its neighbors are occupied, and these effects propagate through the bipartite graph.  

To prove the second part of the lemma, note that $p_{u_i} = \frac{\lam}{1+\lam} q_{u_i}$, and for $u_1, \dots , u_r \in \mathcal E$, $\Pr[ \{u_1, \dots , u_r \} \subseteq I] = (\frac{\lam}{1+ \lam}  )^r \Pr [ \{ u_1, \dots , u_r \} \subseteq U]$, since there are no edges between the $u_i$'s. Then the desired inequality follows from the first part of the lemma.

\begin{proof}[Proof of Theorem~\ref{thm:lowerbound}]
By vertex transitivity, for all $v$, $p_v=\alpha$ and $q_v = \frac{1+\lam}{\lam} \alpha$.  Fix a vertex $v$ and let $Y$  be the number of uncovered neighbors of $v$.  For $u \sim v$ let $Y_u$ be the indicator that $u$ is uncovered.
\begin{align*} \alpha &=  \frac{\lambda}{1+\lambda}\E[(1+\lambda)^{-Y}]\\
                      &= \frac{\lambda}{1+\lambda}\E[(1+\lambda)^{- \sum_{u \sim v}Y_u}]\\
                      &= \frac{\lambda}{1+\lambda}\Big(\alpha + (1-\alpha) \E[(1+\lambda)^{- \sum_{u \sim v}Y_u}|v \notin I]\Big)\,, \text{ hence} \\ 
\frac{\alpha}{\lambda(1-\alpha)} &= \E[(1+\lambda)^{- \sum_{u \sim v}Y_u}|v \notin I]\,.
\end{align*}
Now for $u \sim v$,  let $\tilde Y_u$ be the indicator that $u$ is uncovered, conditioned on the event $\{ v \notin I \}$. For each $u$, $\tilde Y_u$ has a Bernoulli($p$) distribution, where $p =\frac{1+ \lam}{\lam} \frac{\alpha}{1-\alpha}$, and by Lemma~\ref{lem:bipariteCor} applied to $G \setminus v$, the $\tilde Y_u$'s are positively correlated. This gives
\begin{align*}
\frac{\alpha}{\lam (1-\alpha)}    = \E [ (1+\lam)^{-\sum_{u \sim v} \tilde Y_u } ] & > \prod_{u \sim v} \E [ (1+\lam)^{- \tilde Y_u} ] = \left(1-p + \frac{p}{1+\lam}  \right )^d=\left(  \frac{1-2 \alpha}{1-\alpha} \right) ^d\,.
\end{align*}
The function $\frac{\alpha}{\lam (1-\alpha)}$ is increasing in $\alpha$, the function $\left(  \frac{1-2 \alpha}{1-\alpha} \right) ^d$ is decreasing in $\alpha$, and the two functions are equal at $\alpha= \alpha_{T_d}(\lam)$, so we conclude that $\alpha > \alpha_{T_d}(\lam)$. 
\end{proof}

\section{Proof of Theorem~\ref{thm:matchingGeneral}}
\label{sec:matchingproofs}

Recall that we use the notation $M_G(\lam)$ for the matching polynomial of a graph $G$, and let $H$ be a matching drawn from the monomer-dimer model at fugacity $\lam$. 

We  refer to an edge as covered if an incident edge is in the random matching $H$. 
Let $e$ be an edge of $G$ chosen uniformly at random, with an arbitrary left/right orientation chosen at random. In applying the method to matchings we introduce a subtle change of presentation. We now define the \textit{free neighborhood} $C$ to be the subgraph of $G$ containing all the incident edges to $e$ that are not covered by edges outside of both $e$ and its incident edges.  When considering independent sets, the free neighborhood was empty if the random vertex $v$ was in the independent set.  Here the presence or absence  in the matching of $e$ or an edge adjacent to $e$ does not affect $C$.  
Given $e$ and $C$, we use the term \textit{externally uncovered neighbor} to refer to an edge of $C$ incident to $e$.

The possible free neighborhoods $C$ are completely defined by three parameters: $L,R,K\in\{0,1,\dotsc,d-1\}$, counting the number of left and right neighboring edges in $C$ with an endpoint of degree $1$, and the number of triangles formed by $e$ and $C$.  
An example is pictured below.  
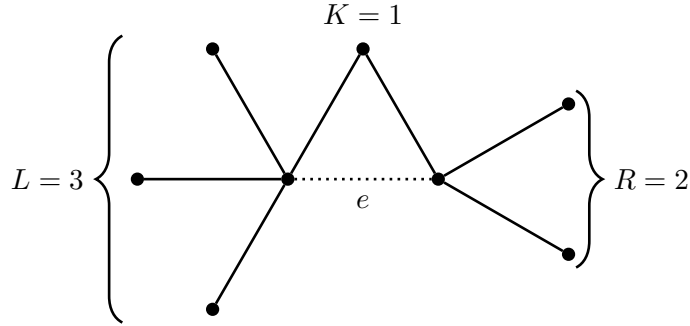
\begin{center}
\begin{tikzpicture}[line width=1pt,every node/.style={circle,fill,draw,inner sep=0,minimum size=4pt}]
	\node (lroot) at (-1,0) {};
	\node (rroot) at ( 1,0) {};
	
	\node (k1) at ($(0,1.732050807568877)$) {};
	\node (l1) at ($(lroot)+(120:2cm)$) {};
	\node (l2) at ($(lroot)+(180:2cm)$) {};
	\node (l3) at ($(lroot)+(240:2cm)$) {};

	\node (r1) at ($(rroot)+(30:2cm)$) {};
	\node (r2) at ($(rroot)+(-30:2cm)$) {};
	
	\node[draw=none, fill=none] (labele) at (0,-0.3) {$e$};
	\node[draw=none, fill=none] (labelk) at (0,2.2) {$K=1$};

	\draw[style=dotted] (lroot) --(rroot);
	
	\draw (lroot) -- (k1);
	\draw (lroot) -- (l1);
	\draw (lroot) -- (l2);
	\draw (lroot) -- (l3);
	
	\draw (rroot) -- (k1);
	\draw (rroot) -- (r1);
	\draw (rroot) -- (r2);
	
	\draw [decorate,decoration={brace,amplitude=10pt}] ($(r1)+(0.1,5pt)$) -- ($(r2)+(0.1,-5pt)$) node [fill=none,draw=none,midway,xshift=1cm] {$R=2$};
	\draw [decorate,decoration={brace,amplitude=10pt,mirror}] ($(l1)+(-1.2cm,5pt)$) -- ($(l3)+(-1.2cm,-5pt)$) node [fill=none,draw=none,midway,xshift=-1cm] {$L=3$};
\end{tikzpicture}
\end{center}
Let $q(i,j,k) = \Pr[L = i, R=j, K=k]$, and denote the matching polynomial for such a free neighborhood by $M_{i,j,k}$, where we can compute 
\[ M_{i,j,k}(\lam) = 1 + (i+j+2k)\lam + \big[k^2 + k(i+j-1) +ij\big]\lam^2 \,. \]
Conditioned on the event that the free neighborhood of $e$ is $C$, the random matching $H$ restricted to $e$ and its incident edges is distributed according to the monomer-dimer model on the graph $C$ with the edge $e$ added; the partition function of this model is $\lam + M_C(\lam)$, with the term $\lam$ corresponding to the event that $e \in H$. 

We can write $\alpha_M:=\alpha^M_G(\lam)$ as the expected fraction of edges incident to $e$ that are in the matching, as each edge in a $d$-regular graph is incident to exactly $2(d-1)$ other edges:
\begin{align*}
\alpha_M & = \frac{2}{dn} \sum_e  \sum_{f \sim e} \frac{1}{2(d-1)} \Pr[f \in H] \\
&= \E \left[ \frac{\lam M_C^\prime(\lam)}{2 (d-1) (\lam + M_C(\lam))}  \right ]  \\
&= \sum_{i,j,k} q(i,j,k) \frac{\lam M^\prime_{i,j,k}(\lam) }{ 2 (d-1) (\lam + M_{i,j,k}(\lam))} \, ,
\end{align*}
where the expectation in the second line is over the random free neighborhood $C$ resulting from the two-part experiment described above.  If we write the expected fraction of occupied neighbors of $e$ in a configuration as $\overline \alpha_M(i,j,k)=\frac{1}{2(d-1)}\frac{\lam M_{i,j,k}'}{\lam + M_{i,j,k}}$,   the above expression can be written $\alpha_M = \sum_{i,j,k}q(i,j,k)\overline\alpha_M(i,j,k)$.

\subsection{The linear program for matchings}

We now introduce additional constraints before optimizing $\alpha_M$ over distributions of free neighborhoods.  We could write multiple expressions for $\alpha_M$, equate them, and solve the maximization problem as we did for independent sets.  Using three expressions for $\alpha_M$ we were able to prove Theorem~\ref{thm:matchingGeneral} for the case $d=3$, in which the optimal distribution is supported on only three values: $q(0,0,0)$, $q(1,1,0)$, $q(2,2,0)$.  But in general we need at least $d-1$ constraints (in addition to the constraint that the $q(i,j,k)'s$ sum to one) as the distribution induced by $K_{d,d}$ is supported on $d$ values. 

Instead, we write, for all $t$, two expressions for the marginal probability that the number of  uncovered neighbors on a randomly chosen side of a random edge is equal to $t$.  We find the two expressions by choosing uniformly: a random edge $e$, a random side left or right, and $f$, a random neighboring edge of $e$  from the given side. 
We first calculate the probability that $e$ has $t$ uncovered neighbors on the side containing $f$, then we calculate the probability that $f$ has $t$ uncovered neighbors on the side containing $e$. 

Given a free neighborhood $C$ with $L=i$, $R=j$, and $K=k$, $e$ can have $0, 1, i+k-1,$ or $i+k$ uncovered left neighbors; an edge $f$ to the left of $e$ can have $0, 1, i+k-2, i+k-1, i+k$, or $i+k+1$ uncovered right neighbors (depending on whether $f$ itself is in the free neighborhood $C$).

Let $\gamma^e_{i,j,k}(t) = \Pr[ e \text { has } t \text{ uncovered left neighbors } | L=i, R=j, K=k]$ and $\gamma^f_{i,j,k}(t) = \Pr[ f \text { has } t \text{ uncovered right neighbors } | L=i, R=j, K=k]$, where $f$ is a uniformly chosen left neighbor of $e$.

\begin{claim}
Let $\beta_t = 1+ t \lam$. Then we have
\begin{align*}\numberthis\label{eq:gamE}
\gam^e_{i,j,k}(t) &= \frac{1}{\lam + M_{i,j,k}} \Bigl( \mathbf 1_{t=0} \cdot \lam + \mathbf 1_{t=1} \cdot [i \lam \beta_{j+k} + k \lam \beta_{j+k-1}] \\
	&\hspace{3.7cm}+ \mathbf 1_{t=i+k} \cdot \beta_j + \mathbf 1_{t=i+k-1} \cdot k \lam \Bigr) \\
\numberthis\label{eq:gamF}
\gam^f_{i,j,k}(t) &= \frac{1}{(d-1)(\lam + M_{i,j,k})} \Big( \mathbf 1_{t=0} \cdot \left[i \lam \beta_{j+k} + k \lam \beta_{j+k-1}\right] \\*
	&\hspace{1cm}+ \mathbf 1_{t=1} \cdot \left[(d-1)\lam + (d-2)(i \lam \beta_{j+k} + k \lam \beta_{j+k-1}  ) \right] \\*
	&\hspace{1cm}+ \mathbf 1_{t=i+k-2} \cdot \left[(i+k-1)k\lam\right] + \mathbf 1_{t=i+k-1} \cdot \left[(d-i-k)k\lam+(i+k)j\lam\right] \\*
	&\hspace{1cm}+ \mathbf 1_{t=i+k} \cdot \left[(d-1-i-k)j\lam+(i+k)\right] + \mathbf 1_{t=i+k+1} \cdot \left[d-1-i-k\right] \Bigr)\,.
\end{align*}
\end{claim}
\begin{proof}
To compute the functions $\gam^e_{i,j,k}(t)$ we consider the following disjoint events: 1) no left edge and no right edge from a triangle is in the matching 2) $e$ is in the matching 3) a left edge is in the matching 4) no left edge is in the matching, but a right edge from a triangle is in the matching.  These events happen with probability $\frac{\beta_j}{\lam + M_{i,j,k}}, \frac{\lam}{\lam + M_{i,j,k}}, \frac{i \lam \beta_{j+k} + k \lam \beta_{j+k-1}}{\lam + M_{i,j,k}},$ and $  \frac{k \lam }{\lam + M_{i,j,k}}$ respectively.  Under these events the number of uncovered neighbors of $e$ is $i+k, 0, 1,$ and $i+k-1$ respectively. This gives \eqref{eq:gamE}.

To compute the functions $\gam^f_{i,j,k}(t)$ we refine the above events to include the possible choices of $f$: $f$ can be an edge outside the free neighborhood with probability $(d-1-i-k)/(d-1)$; an edge in the free neighborhood but not in a triangle with probability $i/(d-1)$; in the free neighborhood and in a triangle with probability $k/(d-1)$. If a left edge is in the matching we choose it as  $f$ with probability $1/(d-1)$, and if a right edge in a triangle is in the matching we choose $f$ adjacent to it with probability $1/(d-1)$.  Computing the number of uncovered neighbors of $f$ in each case gives \eqref{eq:gamF}. 
\end{proof}

We now define a linear program with constraints imposing that the two different ways of writing the marginal probabilities are equal. The marginal probability constraint for $t=d-1$ is redundant and we omit it. To account for the equal chance that $f$ is chosen from the left side of $e$ and the right side of $e$, we average $\gam^f_{i,j,k}(t)$ and $\gam^f_{j,i,k}(t)$, and $\gam^e_{i,j,k}(t)$ and $\gam^e_{j,i,k}(t)$.
\begin{align*}\label{eq:matchlinprogno3}
\alpha_M^* = \max   &\sum_{i,j,k} q(i,j,k) \overline \alpha_M(i,j,k)  \, \, \, \text{ subject to }\\
& q(i,j,k) \ge 0 \, \, \forall \, \,i,j,k\\
&\sum_{i,j,k} q(i,j,k) = 1 \\
&\sum_{i,j,k} q(i,j,k) \frac{1}{2}\Big[\gam^f_{i,j,k}(t)+\gam^f_{j,i,k}(t) -\gam^e_{i,j,k}(t)-\gam^e_{j,i,k}(t)\Big] = 0 \, \, \,\forall \, t= 0, \dots , d-2 \, .
\end{align*}
 Disjoint unions of copies of $K_{d,d}$ are the only graphs that induce a distribution  $q(i,j,k)$ supported on triples with $i=j$ and $k=0$. This gives us a candidate solution to the linear program.

The dual program is
\begin{align*}
\alpha^*_M ={} &\min \, \Lam_p \,  \, \text{ subject to }\\
& \Lam_p - \overline \alpha_M(i,j,k) +\sum_{t=0}^{d-2} \Lam_t \frac{1}{2}\Big[\gam^f_{i,j,k}(t)+\gam^f_{j,i,k}(t) -\gam^e_{i,j,k}(t)-\gam^e_{j,i,k}(t)\Big] \ge 0  \, \, \forall \, \, i,j,k \, .
\end{align*}
To show that $K_{d,d}$ is optimal, we find values for the dual variables $\Lam_0, \dotsc, \Lam_{d-2}$ so that the dual constraints hold with $\Lam_p = \alpha^M_{K_{d,d}}(\lam)$.  To find such values, we solve the system of equations generated by setting equality in the constraints corresponding to $i=j$ and $k=0$ and solve for the variables $\Lambda_t$, $t=0, \dots, d-2$.

With this choice of values for the dual variables, we start by simplifying the form of the dual constraints with a substitution coming from equality in the $(i,j,k)=(0,0,0)$ constraint. The $(0,0,0)$ dual constraint has the simple form
\begin{equation*}
\Lam_0-\Lam_1=\alpha^M_{K_{d,d}}\, .
\end{equation*}
Moreover, observe that from the $\mathbf 1_{t=0}$ and $\mathbf 1_{t=1}$ terms in $\gam_{i,j,k}^e(t)$ and $\gam_{i,j,k}^f(t)$, every dual constraint contains the term
\begin{align*}
\left[\overline \alpha_M(i,j,k) - \frac{\lam}{(\lam+M_{i,j,k})}\right](\Lam_0-\Lam_1)
= \left[\overline \alpha_M(i,j,k) - \frac{\lam}{(\lam+M_{i,j,k})}\right]\alpha^M_{K_{d,d}} \, .
\end{align*}

With this simplification, we multiply through by $2(d-1)(\lam+M_{i,j,k})$ and expand $\overline \alpha_M(i,j,k)$ terms to obtain the following form of the dual constraints:
\begin{align*}
\numberthis
\label{eq:dual2}
\alpha^M_{K_{d,d}}\big[\lam &M'_{i,j,k}+2(d-1)M_{i,j,k}\big]-\lam M'_{i,j,k}\\
&+\Lam_{i+k-2}\cdot(i+k-1)k\lam\\
&+\Lam_{i+k-1}\cdot\left[(d-i-k)k\lam+(i+k)j\lam-(d-1)k\lam\right]\\
&+\Lam_{i+k}\cdot\left[(d-1-i-k)j\lam+i+k-(d-1)\beta_j\right]\\
&+\Lam_{i+k+1}\cdot(d-1-i-k)\\
&+\Lam_{j+k-2}\cdot(j+k-1)k\lam\\
&+\Lam_{j+k-1}\cdot\left[(d-j-k)k\lam+(j+k)i\lam-(d-1)k\lam\right]\\
&+\Lam_{j+k}\cdot\left[(d-1-j-k)i\lam+j+k-(d-1)\beta_i\right]\\
&+\Lam_{j+k+1}\cdot(d-1-j-k) \geq 0 \, .
\end{align*}

The $(i,i,0)$ equality constraints now read
\begin{align}\label{eq:equalityDC}\textstyle
\alpha^M_{K_{d,d}}\beta_i\big(\beta_i+\frac{i\lam}{d-1}\big)-\frac{i\lam\beta_i}{d-1}+\Lam_{i-1}\frac{i^2\lam}{d-1}-\Lam_i\frac{d-1-i+i^2\lam}{d-1}+\Lam_{i+1}\frac{d-1-i}{d-1}=0 \, .
\end{align} 
With this we can write $\Lam_{i+k+1}$ in terms of $\Lam_{i+k}$ and $\Lam_{i+k-1}$, and similarly for $\Lam_{j+k+1}$. 
Substituting this into \eqref{eq:dual2} and dividing by $\lam$ we derive the simplified form of the dual constraints:
\begin{align*}\label{eq:niceDC}
\numberthis\lam\big[&(i-j)^2+2k\big](1-d \alpha^M_{K_{d,d}})\\
&+\Lam_{i+k-2}(i+k-1)k +\Lam_{i+k-1}[k+(i+k)(j-i-2k)]\\
&+\Lam_{i+k}(i+k)(i+k-j)\\
&+\Lam_{j+k-2}(j+k-1)k +\Lam_{j+k-1}[k+(j+k)(i-j-2k)]\\
&+\Lam_{j+k}(j+k)(j+k-i) \geq 0 \, .
\end{align*}
Write $L(i,j,k)$ for the LHS of this inequality.

The marginal constraint for $t=d-1$ was omitted, but we nonetheless introduce $\Lam_{d-1}:=0$ in order to simplify the presentation of the argument.  The $(d-1,d-1,0)$ equality constraint gives $\Lam_{d-2}$ directly:
\begin{align*}
\Lam_{d-2}&=\frac{1}{(d-1)\lam}\left[\lam + (d-1)\lam^2 -\alpha^M_{K_{d,d}} \beta_{d-1}\beta_d\right] \, .
\end{align*}
With $\Lam_{d-1}$, $\Lam_{d-2}$, and the recurrence relation \eqref{eq:equalityDC} the dual variables are fully determined. 
We do not give a closed-form expression for $\Lam_t$ as the values are used in an induction below. 
Using $\Lam_{d-1}$, $\Lam_{d-2}$, and \eqref{eq:equalityDC} suffices for the proof. 

We now reduce the problem of showing that the dual constraints \eqref{eq:niceDC} corresponding to triples $(i,j,k)$ with $k > 0$ or $i\ne j$ hold with strict inequality to showing that a particular function is increasing. We go on to prove this fact in Claims~\ref{claim:Fexp} and~\ref{claim:Fdt}.

Putting $k=0$ into \eqref{eq:niceDC} gives:
\begin{align*}\label{eq:k0dualconst}
\frac{L(i,j,0)}{(j-i)} &= \lam(j-i)(1-d \alpha^M_{K_{d,d}})+i\Lam_{i-1}-i\Lam_{i}-j\Lam_{j-1}+j\Lam_{j}\\
&= F_d(j)-F_d(i)
\end{align*}
where 
\begin{equation}\label{eq:Fdef}
F_d(t) := t\left[\lam(1-d \alpha^M_{K_{d,d}})+\Lam_t-\Lam_{t-1}\right] \, .
\end{equation}
From \eqref{eq:niceDC} we obtain
\begin{align*}
L(i-1,j-1,k+1)-L(i,j,k) = F_d(i+k) - F_d(i+k-1) + F_d(j+k) - F_d(j+k-1).
\end{align*}
Therefore if $F_d(t)$ is strictly increasing, we have $L(i,j,0) > 0$ for $i\ne j$, and $L(i-1,j-1,k+1) > L(i,j,k) > \cdots > L(i+k,j+k,0) \geq 0$.

We first find an explicit expression for $F_d(t)$. Recall that we write $M_{K_{t,t}}$ for the matching polynomial of the graph $K_{t,t}$. 

\begin{claim}\label{claim:Fexp} For all $d\geq 2$ and $1\leq t\leq d-1$,
\begin{align}\label{eq:explicitF}
F_d(t) = \frac{t (d - 1)}{M_{K_{d,d}}}\sum_{\ell=t-1}^{d-2}\frac{(d-1-t)!}{(\ell+1-t)!}\lam^{d-\ell}M_{K_{\ell,\ell}}\,.
\end{align}
\end{claim}
\begin{proof}
We will use the following two facts:
\begin{gather}
\label{eq:LaGuerre}
\MK{d} - \beta_{2d-1} \MK{d-1} + (d-1)^2\lam^2 \MK{d-2} = 0  \\
\label{eq:MKd}
\alpha^M_{K_{d,d}} = \frac{\lam \MK{d-1}}{\MK{d}}  \, .
\end{gather}
The first is a Laguerre polynomial identity, verifiable by hand; the second is a short calculation.  The equality dual constraint \eqref{eq:equalityDC} implies:
\begin{align}\label{eq:Frec}
(d-1-t)F_d(t+1)=(t+1) [ t\lam F_d(t)+(d-1)\lam-(d-1)\alpha^M_{K_{d,d}} \beta_{d+t} ] \,.
\end{align}
We first show that the right hand side of \eqref{eq:explicitF} satisfies the above recurrence relation. Using \eqref{eq:MKd} this amounts to showing that the following expression is equal to zero for all $d\geq 2$ and $1\leq t\leq d-1$:
\[
\Phi_d(t):=(d-1-t)! \Bigg( \sum_{\ell=t}^{d-2}\frac{\lam^{d-\ell}M_{K_{\ell,\ell}}}{(\ell-t)!}-t^2 \sum_{\ell=t-1}^{d-2}\frac{\lam^{d+1-\ell}M_{K_{\ell,\ell}}}{(\ell+1-t)!} \Bigg )-\lam( M_{K_{d,d}}-\beta_{d+t}M_{K_{d-1,d-1}}) \,.
\]
We proceed by induction on $d$. Note that when $d=2$, $\Phi_2(1)$ is easily verified to be zero. Note that
\begin{align*}
\Phi_{d+1}(t) 
&=\lam\Big((d-t)\Phi_d(t) -\MK{d+1}+\beta_{2d+1}\MK{d}- d^2\lam^2\MK{d-1}\Big) \,.
\end{align*}
By the induction hypothesis and \eqref{eq:LaGuerre} the result follows. To complete the proof of the claim it suffices to show that \eqref{eq:explicitF} holds for $t=d-1$. Recalling that
\begin{align*}
\Lam_{d-1} &= 0 \\
\Lam_{d-2} &= \frac{1}{d - 1} + \lam - \frac{\alpha^M_{K_{d,d}}}{(d-1)\lam}\beta_d\beta_{d-1} \,,
\end{align*}
 substituting into \eqref{eq:Fdef}, and using \eqref{eq:LaGuerre} and \eqref{eq:MKd} we have
\begin{align*}
F_d(d-1) 
&= (d-1)\left[\lam(1-d \alpha^M_{K_{d,d}})-\frac{1}{d-1}-\lam+\frac{\alpha^M_{K_{d,d}}}{(d-1)\lam}\beta_d\beta_{d-1}\right]\\
&=\frac{\alpha^M_{K_{d,d}}}{\lam}\beta_{2d-1}-1\\
&=\frac{1}{M_{K_{d,d}}}\left[\beta_{2d-1}M_{K_{d-1,d-1}}-M_{K_{d,d}
}\right]\\
&=\frac{(d-1)^2\lam^2M_{K_{d-2,d-2}}}{M_{K_{d,d}}} \, ,
\end{align*}
verifying \eqref{eq:explicitF} for $t=d-1$. 
\end{proof}

Using Claim \ref{claim:Fexp} we prove the following.

\begin{claim}
\label{claim:Fdt}
$F_d(t)$ is strictly increasing as a function of $t$.
\end{claim}
\begin{proof}
To prove that $F_d(t)$ is increasing, we show that
\begin{align*}
R_d(t) 
	&:=\frac{\MK{d}}{(d-1)}\cdot\frac{F_d(t+1)-F_d(t)}{(d-2-t)!}\\
	&=(t+1)\sum_{\ell=t}^{d-2}\frac{\lambda^{d-\ell}}{(\ell-t)!}\MK{\ell} -t(d-1-t)\sum_{\ell=t-1}^{d-2}\frac{\lambda^{d-\ell}}{(\ell+1-t)!}\MK{\ell}
\end{align*}
is positive for each $t$ with $1\leq t\leq d-2$. We do this by fixing $t$ and inducting on $d$ from $t+2$ upwards. A useful inequality will be $\MK{t} > t\lambda \MK{t-1}$ which comes from only counting matchings of $K_{t,t}$ that use a specific vertex. Iterating this inequality we obtain
\begin{align}\label{eq:crudebound}
\MK{t} > \frac{t!}{\ell !}\lambda^{t-\ell} \MK{\ell} \hbox{ for $0 \leq \ell \leq t-1$}\,.
\end{align}
For the base case of our induction, $d = t+2$, we have $R_d(d-2)=\lam^2 \big[ \MK{d-2}-(d-2)\lambda  \MK{d-3}\big]$ which by \eqref{eq:crudebound} is positive.

For the inductive step we have
\begin{align*}
R_{d+1}(t)=\lambda \bigg[ R_d(t) + \frac{\lambda}{(d-1-t)!}\MK{d-1} - \sum_{\ell=t-1}^{d-2}\frac{t \lambda^{d-\ell}}{(\ell-t+1)!}\MK{\ell} \bigg]
\end{align*}
and so it is sufficient to show
\begin{align}\label{eq:Rdt}
\sum_{\ell=t-1}^{d-2}\frac{t \lambda^{d-\ell}}{(\ell+1-t)!}\MK{\ell} < \frac{\lambda}{(d-1-t)!}\MK{d-1}\,.
\end{align}
We use the inequality \eqref{eq:crudebound} in each term of the sum to see that the LHS of \eqref{eq:Rdt} is less than
\begin{align*}
\sum_{\ell=t-1}^{d-2}\frac{t \ell!\lambda}{(\ell+1-t)!(d-1)!}\MK{d-1}
\end{align*}
and so 
\begin{align*}
\sum_{\ell=t-1}^{d-2}\frac{t \lambda^{d-\ell}}{(\ell+1-t)!}\MK{\ell} &< \sum_{\ell=t-1}^{d-2}\frac{t \ell!\lambda}{(\ell+1-t)!(d-1)!}\MK{d-1}\\
&= \frac{\lambda \MK{d-1}}{(d-1-t)!}\cdot  \sum_{\ell=t-1}^{d-2}\frac{t \ell!(d-1-t)!}{(\ell+1-t)!(d-1)!} \\
&= \frac{\lambda \MK{d-1}}{(d-1-t)!}\cdot \binom{d-1}{t}^{-1}\cdot \sum_{\ell=t-1}^{d-2}\binom{\ell}{t-1}\\
&= \frac{\lambda \MK{d-1}}{(d-1-t)!}\,,
\end{align*}
therefore \eqref{eq:Rdt} holds as required.
\end{proof}

 This completes the proof of dual feasibility and shows our candidate solution to the primal program is optimal.  The uniqueness of the solution follows from two facts. First, strict inequality in the dual constraints outside of the $(i,i,0)$ constraints implies, by complementary slackness, that the support of any optimal solution in the primal is contained in the set of $(i,i,0)$ configurations. Second, the distribution induced by $K_{d,d}$ is the unique distribution satisfying the constraints with such a support. 
 This follows from the fact that $\Lam_i$ is uniquely determined by \eqref{eq:equalityDC} where we have set the $(i,i,0)$ dual constraints to hold with equality, which in turn shows that the relevant $d \times d$ submatrix of the constraint matrix is full rank. This proves Theorem~\ref{thm:matchingGeneral}.

\section{Independent sets and matchings of a given size}
\label{sec:givensize}

Let $i_k(G)$ be the number of independent sets of size $k$ in a graph $G$, and $m_k(G)$ the number of matchings of size $k$.  Kahn~\cite{kahn2001entropy} conjectured that $i_k(G)$ is maximized over $d$-regular, $n$-vertex graphs by $H_{d,n}$ for all $k$ (when $2d$ divides $n$), and Friedland, Krop, and Markstr{\"o}m~\cite{friedland2008number} conjectured the same for $m_k(G)$.  Previous bounds towards these conjectures were given in \cite{carroll2009matchings,ilinca2013asymptotics,perkins2015birthday}; for $d$ fixed and $k$ linear in $n$, all previous bounds were off the conjectured values by a multiplicative factor exponential in $n$.  Here we adapt the method of Carroll, Galvin, and Tetali (and use the above result on the matching polynomial)  to give bounds for both problems that are tight up to a factor of $2 \sqrt{n}$, for all $d$ and all $k$. 

\begin{theorem}
\label{thm:givensize}
For all $d$-regular graphs $G$ on $n$ vertices (where $2d$ divides $n$), 
\begin{align*}
i_k(G) &\le 2 \sqrt{n} \cdot i_k(H_{d,n}) 
\intertext{and}
 m_k(G) &\le 2 \sqrt{n} \cdot m_k(H_{d,n})  \, .
 \end{align*}
\end{theorem}

 We start with a fact about the independence and matching polynomials of $H_{d,n}$. 

\begin{lemma}
\label{lem:lampick}
For all $1 \le k \le n/2$, there exists a $\lambda$ so that
\[ \frac{i_k(H_{d,n}) \lam^k} {P_{H_{d,n}}(\lam)} = \Pr_{H_{d,n}} [|I|=k] > \frac{1}{2 \sqrt{n}}  \]
and a $\lambda$ so that
\[ \frac{m_k(H_{d,n}) \lam^k} {M_{H_{d,n}}(\lam)} = \Pr_{H_{d,n}} [|H|=k] > \frac{1}{2 \sqrt{n}} \, .  \]

\end{lemma}

\begin{proof}
The distribution of the size of a random independent set $I$ drawn from the hard-core model on $H_{d,n}$ is log-concave; that is,
 \[ \Pr_{H_{d,n}}[|I| =j]^2 >  \Pr_{H_{d,n}}[|I| =j+1] \cdot  \Pr_{H_{d,n}}[|I| =j-1] \]
for all $1 < j < n/2$. This follows from two facts: the size distribution of the hard-core model on $K_{d,d}$ is log-concave, and the convolution of two log-concave distributions is again log-concave.  The first fact is simply the calculation
\[ \binom{d}{j} ^2 > \binom{d}{j-1} \binom {d}{j+1} \, . \]
Now choose $\lam$ so that $ \Pr_{H_{d,n}}[|I| =k] =  \Pr_{H_{d,n}}[|I| =k+1]$. Log-concavity then  implies that $ \Pr_{H_{d,n}}[|I| =k]$ is maximal.  Some explicit computations for the variance for a single $K_{d,d}$ give that the variance of $|I|$ is at most $n/8$; then via Chebyshev's inequality, with probability at least $2/3$ the size of $I$ is one of at most $\frac{4}{3} \sqrt n$ values, and thus the largest probability of a single size is greater than $\frac{1}{2 \sqrt n}$.

 The proof for $ m_k(H_{d,n})$ is the same: the variance of the size of a random matching is also at most $n/8$ (see, e.g. \cite{kahn2000normal}), and  log-concavity of the size distribution on $K_{d,d}$ is verified via the inequality
 \[ 
 \binom{d}{j} ^4 j!^2 > \binom{d}{j-1}^2 (j-1)! \binom {d}{j+1}^2 (j+1)!  
 \qedhere
 \] 
\end{proof}

\begin{proof}[Proof of Theorem~\ref{thm:givensize}]
Assume for sake of contradiction that $m_k(G) > 2 \sqrt{n} \cdot m_k(H_{d,n})$.  Choose $\lam$ according to Lemma \ref{lem:lampick}. We have:
\begin{align*}
M_G(\lam)&\ge m_k(G) \lam ^k > 2 \sqrt{n} \cdot m_k(H_{d,n}) \lam^k > M_{H_{d,n}}(\lam) \, ,
\end{align*}
but this contradicts Theorem~\ref{thm:matchingGeneral}. The case of independent sets is identical.   
\end{proof}
The above proof is essentially the same as the proofs in Carroll, Galvin, and Tetali~\cite{carroll2009matchings} with the small observation that $\lam$ can be chosen so that $k$ is the most likely size of a matching (or independent set) drawn from $H_{d,n}$.  The factor $2 \sqrt{n}$ in both cases can surely be improved by using some regularity of the independent set and matchings sequence of a general $d$-regular graph; we leave this for future work.

As a consequence, we prove the asymptotic upper matching conjecture of Friedland, Krop,  Lundow, and Markstr{\"o}m~\cite{friedland2008validations}. Fix $d$ and consider an infinite sequence of $d$-regular graphs $\mathcal G_d = G_{1}, G_{2}, \dots$ where $G_n$ has $n$ vertices. For any $\rho \in [0,1/2]$, the $\rho$-monomer entropy is 
\[ h_{\mathcal G_d} ( \rho) = \sup_{ \{k_n\}} \limsup_{n \to \infty} \frac{ \log m_{k_n} (G_n)}{n} \, , \]
where the supremum is taken over all integer sequences  $ \{ k_n \}$ with $k_n / n \to \rho$.  Let $h_d(\rho) = \lim_{n \to \infty} \frac{\log  m_{\lfloor \rho n \rfloor } ( H_{d,n}) }{ n}$, where the limit is taken over the sequences of integers divisible by $2d$.  Then the conjecture states that for all $\mathcal G_d$ and all $\rho  \in [0,1/2]$, $h_{\mathcal G_d} (\rho) \le h_d(\rho)$. 

To prove this, first assume $\rho > 0$ since for $\rho=0$ the result is trivially true. Assume for the sake of contradiction that  $\limsup \frac{\log m_{k_n} (G_n)}{n} > h_d(\rho) + \eps$ for some $\eps >0$.  Take $N_0$ large enough that for all $n_1 \ge N_0$, divisible by $2d$, $\frac{\log  m_{\lfloor \rho n_1 \rfloor } ( H_{d,n_1}) }{ n_1 } < h_d(\rho) +\eps/2$.  Now take some $n \ge N_0$ with $\frac{\log m_{k_n} (G_n)}{n} > h_d(\rho) + \eps$, and let $n_1 = 2d \cdot \lceil n/(2d) \rceil$.   By Lemma~\ref{lem:lampick}, we choose $\lam$ so that $m_{\lfloor \rho n_1 \rfloor}(H_{d,n_1}) \lam^{\lfloor \rho n_1 \rfloor} > \frac{1}{2 \sqrt{n_1}} M_{H_{d,n_1}}(\lam)$. Note that since $\rho >0$, such $\lam$ is bounded away from $0$ as $n_1 \to \infty$. Then we have
\begin{align*}
 \frac{ \log M_{G_n}(\lam)}{n} \ge \frac{\log m_{k_n}(G_n)\lam^{k_n}}{n} &> \frac{k_n}{n} \log  \lam + h_d(\rho) + \eps\\
 & = \rho \log \lam + h_d(\rho) + \eps +o(1) \text{ as } n \to \infty 
\intertext{and}
\frac{\log M_{K_{d,d}}(\lam)}{2d}  = \frac{ \log M_{H_{d,n_1}}(\lam)}{n_1} &<  \frac{ \log \left( 2 \sqrt{n_1} \cdot m_{\lfloor \rho n_1 \rfloor}(H_{d,n_1}) \lam^{\lfloor \rho n_1 \rfloor} \right ) }{n_1}  \\
 &<  \frac{\log(2 \sqrt {n_1})}{n_1} + \frac{\lfloor \rho n_1 \rfloor}{n_1} \log \lam + h_d(\rho)  + \eps /2 \\
&=  \rho \log \lam  + h_d(\rho) + \eps/2 +o(1) \, , 
\end{align*}
but this contradicts Theorem~\ref{thm:matchingGeneral}.  With the same proof, the analogous statement for independent set entropy holds. 
\section{Conclusions}

To recap, our method consists of writing down a set of constraints on local probabilities in the hard-core or monomer-dimer model that hold for every $d$-regular graph, then optimizing an expression for the occupancy fraction in terms of local probabilities over all distributions that satisfy the constraints.  Verifying that our desired graph is the optimizer involves constructing a feasible solution to the dual linear program. This method allowed us to prove tight bounds on the logarithmic derivative of the partition function in both models. In the case of independent sets the result is a strengthening and an alternate proof of the fact that the independence polynomial is maximized by $K_{d,d}$; in the case of matchings, the corresponding statement about the matching polynomial was itself previously unknown. 

 In both cases our results are neither implied by nor imply conjectures that the numbers of independent sets \cite{kahn2001entropy} and matchings \cite{friedland2008number} of each given size are maximized by $H_{d,n}$; while we improve the known bounds  in both cases, these conjectures remain open. Here we give even stronger conjectures:

\begin{conj}
\label{conj:Ind}
Let $G$ be a $d$-regular, $n$-vertex graph, where $2d$ divides $n$. Then for all $k$, the ratio $\frac{i_k(G)}{i_{k-1}(G)}$ is maximized by $H_{d,n}$.
\end{conj}

\begin{conj}
\label{conj:match}
Let $G$ be a $d$-regular, $n$-vertex graph, where $2d$ divides $n$. Then for all $k$, the ratio $\frac{m_k(G)}{m_{k-1}(G)}$ is maximized by $H_{d,n}$.
\end{conj}

Conjecture \ref{conj:Ind} also appeared in a draft of \cite{perkins2015birthday}.  These conjectures  are stronger than Theorems~\ref{thm:occupy} and \ref{thm:matchingGeneral} and imply the conjectures of \cite{kahn2001entropy} and \cite{friedland2008number}.  The relation to the work here is that Conjectures \ref{conj:Ind} and \ref{conj:match} can be stated as follows: the expected number of neighbors of uniformly random independent set (matching) of size $k$ is minimized by $H_{d,n}$.  Theorems \ref{thm:occupy} and \ref{thm:matchingGeneral} show that such a statement is true when the random independent set (matching) is chosen according to the hard-core model instead of uniformly  over those of a given size.
 
\section*{Acknowledgements}
We thank Prasad Tetali for his helpful comments, Emma Cohen for finding a mistake in an earlier draft of the paper, Laci V\'egh for explaining to us some aspects of duality in linear programming, and Ron Peled for an enjoyable and inspiring discussion of the hard-core model in London.  

\bibliography{IndSets}

\begin{thebibliography}{10}

\bibitem{alon1996independence}
N.~Alon.
\newblock Independence numbers of locally sparse graphs and a ramsey type
  problem.
\newblock {\em Random Structures and Algorithms}, 9(3):271--278, 1996.

\bibitem{bencs2014christoffel}
F.~Bencs.
\newblock Christoffel-{D}arboux type identities for independence polynomial.
\newblock {\em arXiv preprint arXiv:1409.2527}, 2014.

\bibitem{benjamini2011recurrence}
I.~Benjamini and O.~Schramm.
\newblock Recurrence of distributional limits of finite planar graphs.
\newblock In {\em Selected Works of Oded Schramm}, pages 533--545. Springer,
  2011.

\bibitem{bhatnagar2014decay}
N.~Bhatnagar, A.~Sly, P.~Tetali, et~al.
\newblock Decay of correlations for the hardcore model on the $ d $-regular
  random graph.
\newblock {\em Electronic Journal of Probability}, 21, 2016.

\bibitem{boyd2004convex}
S.~Boyd and L.~Vandenberghe.
\newblock {\em Convex optimization}.
\newblock Cambridge university press, 2004.

\bibitem{bregman1973some}
L.~Bregman.
\newblock Some properties of nonnegative matrices and their permanents.
\newblock In {\em Soviet Math. Dokl}, volume~14, pages 945--949, 1973.

\bibitem{carroll2009matchings}
T.~Carroll, D.~Galvin, and P.~Tetali.
\newblock Matchings and independent sets of a fixed size in regular graphs.
\newblock {\em Journal of Combinatorial Theory, Series A}, 116(7):1219--1227,
  2009.

\bibitem{csikvari2014lower}
P.~Csikv{\'a}ri.
\newblock Lower matching conjecture, and a new proof of {S}chrijver's and
  {G}urvits's theorems.
\newblock {\em arXiv preprint arXiv:1406.0766}, 2014.

\bibitem{csikvari2014matchings}
P.~Csikv{\'a}ri.
\newblock Matchings in vertex-transitive bipartite graphs.
\newblock {\em Israel Journal of Mathematics}, 215(1):99--134, 2016.

\bibitem{fill2001stochastic}
J.~A. Fill and M.~Machida.
\newblock Stochastic monotonicity and realizable monotonicity.
\newblock {\em Annals of probability}, pages 938--978, 2001.

\bibitem{friedland2008validations}
S.~Friedland, E.~Krop, P.~H. Lundow, and K.~Markstr{\"o}m.
\newblock On the validations of the asymptotic matching conjectures.
\newblock {\em Journal of Statistical Physics}, 133(3):513--533, 2008.

\bibitem{friedland2008number}
S.~Friedland, E.~Krop, and K.~Markstr{\"o}m.
\newblock On the number of matchings in regular graphs.
\newblock {\em The Electronic Journal of Combinatorics}, 15(1):R110, 2008.

\bibitem{galvin2014three}
D.~Galvin.
\newblock Three tutorial lectures on entropy and counting.
\newblock {\em arXiv preprint arXiv:1406.7872}, 2014.

\bibitem{galvin2004weighted}
D.~Galvin and P.~Tetali.
\newblock On weighted graph homomorphisms.
\newblock {\em DIMACS Series in Discrete Mathematics and Theoretical Computer
  Science}, 63:97--104, 2004.

\bibitem{heilmann1972theory}
O.~J. Heilmann and E.~H. Lieb.
\newblock Theory of monomer-dimer systems.
\newblock {\em Communications in Mathematical Physics}, 25(3):190--232, 1972.

\bibitem{ilinca2013asymptotics}
L.~Ilinca and J.~Kahn.
\newblock Asymptotics of the upper matching conjecture.
\newblock {\em Journal of Combinatorial Theory, Series A}, 120(5):976--983,
  2013.

\bibitem{kahn2000normal}
J.~Kahn.
\newblock A normal law for matchings.
\newblock {\em Combinatorica}, 20(3):339--391, 2000.

\bibitem{kahn2001entropy}
J.~Kahn.
\newblock An entropy approach to the hard-core model on bipartite graphs.
\newblock {\em Combinatorics, Probability and Computing}, 10(03):219--237,
  2001.

\bibitem{kahn2002entropy}
J.~Kahn.
\newblock Entropy, independent sets and antichains: a new approach to
  {D}edekind's problem.
\newblock {\em Proceedings of the American Mathematical Society},
  130(2):371--378, 2002.

\bibitem{kopparty2011homomorphism}
S.~Kopparty and B.~Rossman.
\newblock The homomorphism domination exponent.
\newblock {\em European Journal of Combinatorics}, 32(7):1097--1114, 2011.

\bibitem{lelarge2015counting}
M.~Lelarge.
\newblock Counting matchings in irregular bipartite graphs and random lifts.
\newblock In {\em Proceedings of the Twenty-Eighth Annual ACM-SIAM Symposium on
  Discrete Algorithms}, pages 2230--2237. SIAM, 2017.

\bibitem{lubetzky2014replica}
E.~Lubetzky and Y.~Zhao.
\newblock On replica symmetry of large deviations in random graphs.
\newblock {\em Random Structures \& Algorithms}, 2014.

\bibitem{perkins2015birthday}
W.~Perkins.
\newblock Birthday inequalities, repulsion, and hard spheres.
\newblock {\em Proceedings of the American Mathematical Society},
  144(6):2635--2649, 2016.

\bibitem{radhakrishnan20036}
J.~Radhakrishnan.
\newblock Entropy and counting.
\newblock {\em Computational Mathematics, Modelling and Algorithms}, page 146,
  2003.

\bibitem{ruozzi2012bethe}
N.~Ruozzi.
\newblock The {B}ethe partition function of log-supermodular graphical models.
\newblock In {\em Advances in Neural Information Processing Systems}, pages
  117--125, 2012.

\bibitem{schulman2015symbolic}
L.~J. Schulman, A.~Sinclair, and P.~Srivastava.
\newblock Symbolic integration and the complexity of computing averages.
\newblock In {\em Foundations of Computer Science (FOCS), 2015 IEEE 56th Annual
  Symposium on}, pages 1231--1245. IEEE, 2015.

\bibitem{shearer1995independence}
J.~B. Shearer.
\newblock On the independence number of sparse graphs.
\newblock {\em Random Structures \& Algorithms}, 7(3):269--271, 1995.

\bibitem{sinclair2014lee}
A.~Sinclair and P.~Srivastava.
\newblock Lee--{Y}ang theorems and the complexity of computing averages.
\newblock {\em Communications in Mathematical Physics}, 329(3):827--858, 2014.

\bibitem{van1994percolation}
J.~Van~den Berg and J.~Steif.
\newblock Percolation and the hard-core lattice gas model.
\newblock {\em Stochastic Processes and their Applications}, 49(2):179--197,
  1994.

\bibitem{weitz2006counting}
D.~Weitz.
\newblock Counting independent sets up to the tree threshold.
\newblock In {\em Proceedings of the thirty-eighth annual ACM symposium on
  Theory of computing}, pages 140--149. ACM, 2006.

\bibitem{zhao2010number}
Y.~Zhao.
\newblock The number of independent sets in a regular graph.
\newblock {\em Combinatorics, Probability and Computing}, 19(02):315--320,
  2010.

\bibitem{zhao2011bipartite}
Y.~Zhao.
\newblock The bipartite swapping trick on graph homomorphisms.
\newblock {\em SIAM Journal on Discrete Mathematics}, 25(2):660--680, 2011.

\end{thebibliography}
\bibliographystyle{abbrv}

\end{document}